\documentclass[11pt]{amsart}

\usepackage[margin=1in]{geometry}

\usepackage{amsmath, amscd, amsthm, amssymb,hyperref}

\def\C{{\mathbf C}}
\def\R{{\mathbf R}}
\def\Z{{\mathbf Z}}
\def\Q{{\mathbf Q}}
\def\A{{\mathbf A}}

\newtheorem{theorem}{Theorem}[section]

\newtheorem{lemma}[theorem]{Lemma}

\newtheorem{proposition}[theorem]{Proposition}

\newtheorem{corollary}[theorem]{Corollary}

\newtheorem{claim}[theorem]{Claim}

\theoremstyle{definition}
\newtheorem{definition}[theorem]{Definition}

\theoremstyle{remark}

\newcommand{\mm}[4]{\left(\begin{smallmatrix} #1 & #2\\ #3 & #4\end{smallmatrix}\right)}

\DeclareMathOperator{\tr}{tr}

\DeclareMathOperator{\Sp}{Sp}

\DeclareMathOperator{\SL}{SL}
\DeclareMathOperator{\GL}{GL}

\DeclareMathOperator{\diag}{diag}

\begin{document}
	\title{Automatic convergence for Siegel modular forms}
	\author{Aaron Pollack}
	\address{Department of Mathematics\\ University of California San Diego\\ La Jolla, CA USA}
	\email{apollack@ucsd.edu}
	\thanks{Funding information: AP has been supported by the NSF via grant numbers 2101888 and 2144021.}
	
	\begin{abstract} Bruinier and Raum, building on work of Ibukiyama-Poor-Yuen, have studied a notion of \emph{formal Siegel modular forms}.  These objects are formal sums that have the symmetry properties of the Fourier expansion of a holomorphic Siegel modular form.  These authors proved that formal Siegel modular forms necessarily converge absolutely on the Siegel half-space, and thus are the Fourier expansion of an honest Siegel modular form.  The purpose of this note is to give a new proof of the cuspidal case of this ``automatic convergence" theorem of Bruinier-Raum.  We use the same basic ideas in a separate paper to prove an automatic convergence theorem for cuspidal quaternionic modular forms on exceptional groups.
	\end{abstract}
	
	\maketitle
	
	\setcounter{tocdepth}{1}
	\tableofcontents
	\section{Introduction} 
	In the recent preprint \cite{pollackAutConvExc}, we proved that the cuspidal quaternionic modular forms on the groups of type $F_4$ and $E_n$, $n = 6,7,8$ have an algebraic structure, defined in terms of Fourier coefficients.  One key step in the proof is to prove what can be called an \emph{automatic convergence theorem}, which asserts that certain infinite series that ``look" like the Fourier expansion of a modular form necessarily \emph{are} the Fourier expansion of a modular form.  The proof in \cite{pollackAutConvExc} is rather technical.  In this note, we illustrate the basic method of \cite{pollackAutConvExc} in the case of Siegel modular forms.  While some key pieces of the argument in \cite{pollackAutConvExc} have no analogue in the Siegel modular case, we believe that the main flavor of the argument remains.
	
	Automatic convergence theorems have a history, going back to Ibukiyama-Poor-Yuen \cite{IPY}, and then developed in work of Bruinier \cite{bruinierGenus2}, Raum \cite{raumGenus2}, Bruinier-Raum \cite{bruinerRaum2015, bruinierRaum2024} and Xia \cite{xiaUnitary}.  The automatic convergence theorem we give here is the cuspidal case of a result of \cite{bruinerRaum2015}, but the argument we give is new.
	
	\section{Fourier-Jacobi expansion of Siegel modular forms}
	In this section, we review the Fourier and Fourier-Jacobi expansions of Siegel modular forms.  

\subsection{The Fourier expansion}
	Let $\mathcal{H}_n = \{Z = X+iY \in M_n(\C): Z = Z^t, Y > 0\}$ be the Siegel upper half space of degree $n$.  The group $\Sp_{2n}(\R)$ acts on $\mathcal{H}_n$ by fractional linear transformations, $g \cdot Z = (aZ+b)(cZ+d)^{-1}$ if $g = \mm{a}{b}{c}{d}$.  Let $K_\infty$ denote the stabilizer of $i 1_n$ for this action.  Then $K_\infty \approx U(n)$, the unitary group.  If $Z \in \mathcal{H}_n$ and $g \in \Sp_{2n}(\R)$, $g = \mm{a}{b}{c}{d}$, let $J(g,Z) = cZ + d \in \GL_n(\C)$ and $j(g,Z) = \det(J(g,Z))$.  
	
	Fix an integer $\ell \geq 0$.  If $\varphi: \Sp_{2n}(\Q) \backslash \Sp_{2n}(\A) \rightarrow \C$ is an automorphic form, we say that $\varphi$ corresponds to a Siegel modular form of weight $\ell$ if, for every $g_f \in \Sp_{2n}(\A_f)$, the function $f_{\varphi,g_f}: \Sp_{2n}(\R) \rightarrow \C$ given by $f_{\varphi,g_f}(g_\infty) = j(g_\infty,i)^{\ell} \varphi(g_f g_\infty)$ descends to $\mathcal{H}_n$ and gives a holomorphic function there.
	
	Let $S_n(\Z)$, $S_n(\Q)$, etc denote the $n \times n$ symmetric matrices with integer or rational entries etc.  Let $S_n(\Z)^\vee$ denote the half-integral symmetric matrices.  If $T \in S_n(\R)$ is positive semi-definite, define $\mathcal{W}_{\ell,T}: \Sp_{2n}(\R) \rightarrow \C$ as $\mathcal{W}_{\ell,T}(g) = j(g,i)^{-\ell} e^{2\pi i (g \cdot i)}$.
	
	Suppose $\varphi$ is a cuspidal automorphic form, corresponding to a Siegel modular form of weight $\ell$.  Then $\varphi$ has a Fourier expansion
	\begin{equation}\label{eqn:SMFFE}\varphi(g_f g_\infty) = \sum_{T \in S_n(\Q), T > 0}a_T(g_f) \mathcal{W}_{\ell,T}(g_\infty).\end{equation}
	Here $a_T: \Sp_{2n}(\A_f) \rightarrow \C$ is a locally constant function, called the $T$-Fourier coefficient of $\varphi$.  Let $\psi: \Q \backslash \A \rightarrow \C^\times$ denote the standard additive character.  If $X \in S_n$, let $n(X) = \mm{1}{X}{}{1} \in \Sp_{2n}$.  The Fourier coefficients $a_T$ satisfy $a_T(n(X)g_f) = \psi((T,X))a_T(g_f)$ for every $g_f \in \Sp_{2n}(\A_f)$ and every $X \in S_n(\A_f)$.  Here $(T,X) = \frac{1}{2} \tr(TX+XT) = \tr(TX)$.
	
	\begin{definition} Suppose $T \geq 0$ and $a_T: G(\A_f) \rightarrow \C$ is a locally constant function satisfying
		\begin{enumerate}
			\item $a_T(n(X)g_f) = \psi((T,X))a_T(g_f)$ for every $g_f \in \Sp_{2n}(\A_f)$ and every $X \in S_n(\A_f)$;
			\item there exists a compact open subgroup $U$ so that $a_T$ is right-invariant by $U$.
		\end{enumerate}
		In this case, we say that $a_T$ is Siegel-Whittaker function, or $T$-Siegel Whittaker function if we need to specify the element $T$.
	\end{definition}
	
	If $r \in \GL_n$, let $m_n(r) = \diag(r, \,^tr^{-1}) \in \Sp_{2n}$.  The group $\GL_n$ acts on $S_n$ on the right as $T \cdot r = r^t T r$. The automorphy of $\varphi$ implies that $a_{T}(m_n(\gamma)g_f) = \det(\gamma)^{-\ell} a_{T \cdot \gamma}(g_f)$ for every $g_f \in \Sp_{2n}(\A_f)$ and every $\gamma \in \GL_n(\Q)$.
	
	Let $\widetilde{\Sp}_{2n}(\R)$ be the double cover of $\Sp_{2n}(\R)$.  For $g \in \widetilde{\Sp}_{2n}(\R)$, let $j_{1/2}(g,Z)$ denote the canonical squareroot of $j(\overline{g},Z)$, where $\overline{g}$ is the image of $g$ in $\Sp_{2n}(\R)$.  If $\ell \in 2^{-1} \Z$ and $g \in \widetilde{\Sp}_{2n}(\R)$, one can define $\mathcal{W}_{\ell,T}(g) = j_{1/2}(g,i)^{-2\ell} e^{2\pi i (\overline{g} \cdot i)}.$  
	
	Let now $\widetilde{\Sp}_{2n}(\A)$ denote the metaplectic double cover of $\Sp_{2n}(\A)$.  The group $\Sp_{2n}(\Q)$ splits uniquely into $\widetilde{\Sp}_{2n}(\A)$.  One can define automorphic forms on the metaplectic double cover.   We say that $\varphi$ corresponds to a Siegel modular form of weight $\ell \in 2^{-1} \Z$ if, for every $g_f \in \widetilde{\Sp}_{2n}(\A_f)$, the function $f_{\varphi,g_f}: \widetilde{\Sp}_{2n}(\R) \rightarrow \C$ given by $f_{\varphi,g_f}(g_\infty) = j_{1/2}(g_\infty,i)^{2\ell} \varphi(g_f g_\infty)$ descends to $\mathcal{H}_n$ and gives a holomorphic function there.  These half-integral weight automorphic forms again have a Fourier expansion of the form \eqref{eqn:SMFFE}.
	
\subsection{The Fourier-Jacobi expansion}
 We will recall, without proof, the Fourier-Jacobi expansion of Siegel modular forms along the Klingen parabolic subgroup $Q = M_Q N_Q$ of $\Sp_{2n}$.  Here $N_Q$ is the unipotent radical of $Q$ and $M_Q$ is its standard Levi subgroup.  Let $X = N_Q \cap M_P$, where $M_P \simeq \GL_n$ is the Levi of the Siegel parabolic subgroup.  For every nonzero rational number $t$, there is a Weil representation $\omega_{t}$ of $J_{n-1}(\A) := N_Q(\A) \rtimes \widetilde{\Sp}_{2n-2}(\A)$ on $S(X(\A))$, the Schwartz-Bruhat space of $X(\A)$.  This representation has the center $Z(N_Q)$ of $N_Q$ acting by $\psi(tz)$, for $z \in Z(N_Q)(\A)$.
 
 If $\phi \in S(X(\A_f))$, $T' \in S_{n-1}(\Q)$, $a_{\diag(t,T')}$ is a $T=\diag(t,T')$ Siegel Whittaker function, $g_f \in \Sp_{2n}(\A_f)$ and $r_f \in \widetilde{\Sp}_{2n-2}(\A_f)$, define
 \[ a_{t,T'}(r_f,g_f;\phi) = \int_{X(\A_f)}{a_{\diag(t,T')}(x\overline{r_f} g_f) (\omega_{-t}(r_f)\phi)(x)\,dx}\]
 where $\Sp_{2n-2}$ is embedded in $\Sp_{2n}$ inside of $M_Q$.
 
 \begin{proposition}\label{prop:Qsymm} Suppose $\varphi$ corresponds to a cuspidal Siegel modular form of weight $\ell$, $g_f \in \Sp_{2n}(\A_f)$, and $\phi \in S(X(\A_f))$.  If $t > 0$, then the $a_{t,T'}(r_f,g_f;\phi)$ are the Fourier coefficients of a cuspidal weight $\ell-\frac{1}{2}$ modular form on $\widetilde{\Sp}_{2n-2}(\A)$ as $T'$ varies over the positive-definite elements of $S_{n-1}(\Q)$.
 \end{proposition}

The proof of the proposition is to first define a theta function $\Theta_\phi(nr)$ on $J_{n-1}(\A)$, 
\[\Theta_\phi(nr) = \sum_{\xi \in X(\Q)}{ \omega_{-t}(nr)(\phi \otimes \phi_\infty)(\xi)}.\]
Here $\phi_\infty$ is an appropriately chosen Gaussian on $X(\R)$.  Then, one obtains an automorphic form on $\widetilde{\Sp}_{2n-2}$ as
\[\mathrm{FJ}_\phi(r_f,g_f) = \int_{[N_Q]}{\Theta_{\phi}(nr)\varphi(nrg_f)\,dn}.\]
One calculates the Fourier expansion of this automorphic form, and finds that it corresponds to a holomorphic modular form of weight $\ell - \frac{1}{2}$ and has Fourier coefficients the $a_{t,T'}(r_f,g_f;\phi)$.

\begin{definition} Suppose $\{a_T\}_{T}$ are a collection of Siegel-Whittaker functions.  We say that the collection satisfies the \textbf{$P$-symmetries} if $a_{T}(m_n(\gamma)g_f) = \det(\gamma)^{-\ell} a_{T \cdot \gamma}(g_f)$ for every $g_f \in \Sp_{2n}(\A_f)$ and every $\gamma \in \GL_n(\Q)$.  We say the collection satisfies the \textbf{$Q$-symmetries} if the conclusion of Proposition \ref{prop:Qsymm} holds.  That is, for every $t > 0$ and $\phi \in S(X(\A_f))$, the $a_{t,T'}(r_f,g_f;\phi)$ are the Fourier coefficients of a cuspidal weight $\ell-\frac{1}{2}$ modular form on $\widetilde{\Sp}_{2n-2}(\A)$ as $T'$ varies over the positive-definite elements of $S_{n-1}(\Q)$.
\end{definition}

\section{Converse theorem}
If $\varphi$ is corresponds to a cuspidal Siegel modular form on $\Sp_{2n}$ of weight $\ell$, then its Fourier coefficients $a_T$ satisfy the $P$ and $Q$ symmetries.  Additionally, the Fourier coefficients are uniformly smooth, in the sense there is an open compact subgroup $U$ of $\Sp_{2n}(\A_f)$ so that the $a_T$ are right $U$ invariant for every $T$.  Finally, the $a_T(g_f)$ grow polynomially in the norm $||T|| = (T,T)^{1/2}$ for every $g_f \in \Sp_{2n}(\A_f)$.  (In fact, they satisfy more specific, tighter bounds.)

The converse is also true.
\begin{proposition} Suppose the $\{a_T\}_T$ are a collection of Siegel Whittaker functions that satisfy the $P$ and $Q$ symmeties, are uniformly smooth, and grow polynomially with $||T||$.  Then the sum
	\[\Psi(g_fg_\infty) = \sum_{T \in S_n(\Q), T>0} a_T(g_f) W_{\ell,T}(g_\infty)\]
converges absolutely.  It defines a cuspidal automorphic form on $\Sp_{2n}$ that corresponds to a Siegel modular form of weight $\ell$.
\end{proposition}
\begin{proof}[Proof sketch]
One can show that the sum converges absolutely by the fact that the $a_T$ grow polynomially.  The result is a smooth, moderate growth, $\mathcal{Z}(\mathfrak{sp}_{2n})$-finite function on $\Sp_{2n}(\A)$.  The fact that the $a_{T}$ satisfies the $P$ symmetries implies that $\Psi(g)$ is left $P(\Q)$-invariant.  

Let $Q^1$ denote the derived group of $Q$.   The fact that the $a_{T}$ satisfies the $Q$ symmetries implies that $\Psi(g)$ is left $Q^1(\Q)$-invariant on $\Sp_{2n}(\A_f) \times (Q^1(\R) K_\infty)$.  Indeed, suppose $t > 0$.  Let $\Psi_t(g) = \int_{[Z(N_Q)]}{\psi(-tz)\Psi(zg)\,dz}$.  Then one can reconstruct $\Psi_{t}(g)$ from the $a_{t,T'}(r_f,g_f,\phi)$ and theta functions.  Specifically, set
\[\Psi_{t,\phi}(r;g_f) = \sum_{T' \in S_{n-1}(\Q), T' > 0}{a_{t,T'}(r_f,g_f;\phi) \mathcal{W}_{\ell-\frac{1}{2},T'}(r_\infty)}\]
which is a cuspidal automorphic form on $\widetilde{\Sp}_{2n-2}(\A)$ by the fact that the $a_{T}$ satisfy the $Q$-symmetries.  Then
\[\Psi_t(nr g_f ) = \sum_{\alpha}{\Psi_{t,\phi_\alpha}(r;g_f) \Theta_{\phi_\alpha^\vee}(nr)}.\]
Here $\phi_\alpha$ is a basis of $S(X(\A_f))$ and $\phi_\alpha^\vee$ is the dual basis.  This expansion shows that $\Psi_t(qkg_f)$ is left-invariant under $Q^1(\Q)$ for every $t \in \Q^\times_{>0}$, $q \in Q^1(\Q)$, $g_f \in \Sp_{2n}(\A_f)$ and $k  \in K_\infty$.  As $\Psi = \sum_t \Psi_t$, one obtains that $\Psi$ is left $Q^1(\Q)$-invariant on $\Sp_{2n}(\A_f) \times (Q^1(\R) K_\infty)$. 

But $\Psi$ corresponds to a holomorphic function on $\mathcal{H}_n$, so by the identity theorem for holomorphic functions, one can check that $\Psi$ is left $Q^1(\Q)$-invariant on all of $\Sp_{2n}(\A)$.  As $\Psi$ is both left invariant for $P(\Q)$ and $Q^1(\Q)$, it is left $\Sp_{2n}(\Q)$-invariant.
\end{proof}

The aim of the rest of note is prove that assumption that the $a_T$ grow polynomially with $T$ is unnecessary.
\begin{theorem}[Automatic convergence for Siegel modular forms] \label{thm:autConv} Suppose the $\{a_T\}_T$ are a collection of Siegel Whittaker functions that satisfy the $P$ and $Q$ symmeties, and are uniformly smooth.  Then the $a_T$ grow polynomially with $T$.  Consequently, every such collection is the set of Fourier coefficients of an honest cuspidal Siegel modular form of weight $\ell$.
\end{theorem}

The proof of Theorem \ref{thm:autConv} combines some reduction theory with a ``quantitative Sturm bound".  
\section{Reduction theory and the Quantitative Sturm bound}
In this section, we review some reduction theory and prove the quantitative Sturm bound.

\subsection{Reduction theory}
There are two results we will need from reduction theory.  The first involves Minkowski reduction theory for $\GL_n$.  It is standard.  See, e.g., \cite[equation (1.22)]{andrianovBook}.
\begin{theorem}\label{thm:minkowski1} There is a positive constant $C_n$ with the following property: Suppose $T \in S_n(\R)$ is positive-definite.  Then there is $\gamma \in \GL_n(\Z)$ so that $T \cdot \gamma$ has $(11)$ entry at most $C_n \det(T)^{1/n}$.
\end{theorem}

The second involves a Siegel set for the action of $\Sp_{2n}(\Z)$ on the Siegel upper half-plane.  One can see, e.g., \cite[Theorem 1.16]{andrianovBook}.
\begin{theorem}\label{thm:Siegel1} There is a positive constant $\epsilon_n$ with the following property: Suppose $Z \in \mathcal{H}_n$. Then there is $\gamma \in \Sp_{2n}(\Z)$ so that $Im(\gamma \cdot Z) > \epsilon_n 1_n$.
\end{theorem}

Set $\mathcal{S}(\epsilon_n) = \{g \in \Sp_{2n}(\R): Im(g \cdot i) > \epsilon_n 1_n\}$, and let $\widetilde{\mathcal{S}}(\epsilon_n)$ denote the inverse image of $\mathcal{S}(\epsilon_n)$ in $\widetilde{\Sp}_{2n}(\R)$.  Let $K_f = \prod_{p < \infty}{\Sp_{2n}(\Z_p)}$.  Let $\widetilde{K}_f$ be the inverse image of $K_f$ in $\widetilde{Sp}_{2n}(\A_f)$.
\begin{corollary}\label{cor:adelicRedTh} Suppose $g \in \Sp_{2n}(\A)$.  Then there is $\gamma \in \Sp_{2n}(\Q)$ so that $\gamma g \in \mathcal{S}(\epsilon_n) K_f$.  Likewise, if $g \in \widetilde{\Sp}_{2n}(\A)$, then there is $\gamma \in \Sp_{2n}(\Q)$ so that $\gamma g \in \widetilde{\mathcal{S}}(\epsilon_n) \widetilde{K}_f$.
\end{corollary}
\begin{proof} By approximation, one has $\Sp_{2n}(\A_f) = \Sp_{2n}(\Q) K_f$.  Thus we can write $g = \gamma_1 g_1$ with $g_{1,f} \in K_f$.  By Theorem \ref{thm:Siegel1}, there is $\gamma_2 \in \Sp_{2n}(\Z)$ so that $\gamma_{2,\infty} g_{1,\infty} \in \mathcal{S}(\epsilon_n)$.  The corollary follows.
\end{proof}

\subsection{Quantitative Sturm bound}
Recall that the classical Sturm bound result says that if the first several Fourier coefficients of a holomorphic modular form of some weight $\ell$ are $0$, then the modular form is identitcally $0$.  By a quantitative Sturm bound we mean a result of the form ``if the first several Fourier coefficients of a holomorphic modular form of some weight $\ell$ are at most $\epsilon \geq 0$, then the automorphic function $\varphi$ is bounded by a constant times $\epsilon$".  In particular, all the Fourier coefficients of $\varphi$ are bounded by a constant times $\epsilon$.

We explain the proof of such a result on $\Sp_{2n}$ here.  We begin with a few simple lemmas.
\begin{lemma}\label{lem:TMbound} Suppose $X  \in \R_{> 0}$.  
	\begin{enumerate}
		\item The number of $T \in S_{n}(\Z)^\vee$ positive-definite with $\tr(T) \leq X$ is bounded by $2^{n(n-1)/2} X^{n(n+1)/2}$.
		\item If $M \in \Z_{\geq 1}$, the number of $T \in M^{-1} S_{n}(\Z)^\vee$ positive-definite with $\tr(T) \leq X$ is bounded by $2^{n(n-1)/2} M^{n(n+1)/2} X^{n(n+1)/2}$.
	\end{enumerate}
\end{lemma}
\begin{proof} We prove the first statement of the lemma.  The second follows from the first.  Thus suppose $T \in S_n(\Z)^\vee$ is positive-definite and $\tr(T) \leq X$.  Then each diagonal entry $T_{ij}$ is an integer between $1$ and $X$.  For the off diagonal entries, we have $T_{ii} T_{jj} - \frac{T_{ij}^2}{4} > 0$.  Thus 
	\[|T_{ij}| < 2\sqrt{T_{ii} T_{jj}} \leq T_{ii} + T_{jj}.\]
So, $|T_{ij}| \leq T_{ii}+T_{jj}-1 \leq X-1$.  Thus there at most $2X$ possibilities for $T_{ij}$.
\end{proof}

Suppose $r> 0, N >0$.  Finding the critical point of the function $f(v) = v^N e^{-r v}$ for $v >0$ gives:
\begin{lemma}\label{lem:powerExp} One has $v^N e^{-rv} \leq (N/r)^N e^{-N}$ for $v \geq 0$.
\end{lemma}

We now bound some infinite sums.
\begin{lemma} \label{lem:InfSums} Suppose $R > 0$ and $M \in \Z_{\geq 1}$. For $y \in S_n(\R)$ positive-definite and $\ell \geq 0$, define
	\[S_{\ell}(M,y) = \sum_{T \in M^{-1} S_n(\Z)^\vee}{\det(T)^{\ell/2} \det(y)^{\ell/2} e^{-2\pi (T,y)}}\]
and
	\[T_{\ell}(R,M,y) = \sum_{T \in M^{-1} S_n(\Z)^\vee, \det(T) \geq R}{\det(T)^{\ell/2} \det(y)^{\ell/2} e^{-2\pi (T,y)}}.\]
If $1>\mu > 0$ and $y \geq \mu 1_n$, then there are positive constants $D_{\ell,n,\mu}$ and  $\alpha$ so that 
\[S_{\ell}(M,y) \leq D_{\ell,n,\mu} M^{\alpha}\]
for all $M \geq 1$.  Likewise, 
\[T_{\ell}(R,M,y) \leq D_{\ell,n,\mu} M^{\alpha} e^{- \pi \mu n R^{1/n}/2}.\]
\end{lemma}
\begin{proof} First note that, by the AM-GM inequality,
	\begin{equation} \label{eqn:AMGM}\det(Ty) = \det(y^{1/2} T y^{1/2}) \leq \left(\frac{1}{n}\tr(y^{1/2}Ty^{1/2})\right)^{n} = n^{-n} (T,Y)^n.\end{equation}
Thus $\det(TY)^{\ell/2} e^{-\pi(T,Y)} \leq n^{-n\ell/2} (T,Y)^{\ell n/2} e^{-\pi(T,Y)}$, which by Lemma \ref{lem:powerExp} is bounded by 
\[C_{\ell,n}:=\left(\frac{\ell n}{2 \pi}\right)^{\ell n /2} n^{-n\ell/2} e^{-\ell n /2}.\]
Thus 
\[S_{\ell}(M,y) \leq C_{\ell,n}  \sum_{T \in M^{-1} S_n(\Z)^\vee}{ e^{-\pi (T,y)}} \leq C_{\ell,n} \sum_{T \in M^{-1} S_n(\Z)^\vee}{ e^{-\pi \mu \tr(T)}}.\]
Applying Lemma \ref{lem:TMbound}, one obtains
\[S_{\ell}(M,y) \leq C_{\ell,n} 2^{n(n-1)/2} \sum_{k \geq 1}{ k^{n(n+1)/2} e^{-\pi \mu k/M}}.\]
Applying Lemma \ref{lem:powerExp} again, summing a geometric series, and applying the inequality $\frac{e^{-r}}{1-e^{-r}} \leq r^{-1}$ for $r > 0$ gives the bound on $S_{\ell}(M,y)$.

To bound $T_{\ell}(R,M,y)$, first note that if $\det(T) \geq R$ and $y > \mu 1_n$, then using \eqref{eqn:AMGM}, one sees that $(T,y)$ is bounded below by $n R^{1/n}\mu$.  Thus, to bound $T_{\ell,M,y}$ it suffices to bound 
\[\sum_{k \geq R^{1/n} n M}{ k^{n(n+1)/2} e^{-\pi \mu k/M}}.\]
Arguing as above, one gets the claim.
\end{proof}

If $\varphi$ corresponds to a Siegel modular form of weight $\ell$ with Fourier coefficients $a_T(g_f)$, let $\beta_T(g_f) = \det(T)^{-\ell/2} a_T(g_f)$ be the normalized Fourier coefficients.
\begin{lemma}\label{lem:FCbound} Suppose $|\varphi(g)| \leq L$ for all $g \in \Sp_{2n}(\A)$.  Then $|\beta_T(g_f)| \leq e^{2\pi n} L$ for all $T > 0$ and all $g_f \in \Sp_{2n}(\A_f)$.
\end{lemma}
\begin{proof} This is standard.  From $|\varphi(g)| \leq L$ for all $g$, one obtains by integration that 
	\[|\beta_T(g_f)| \det(T)^{\ell/2} |\mathcal{W}_{\ell,T}(g_\infty)| \leq L\]
	 for all $g_f$, all $g_\infty$.   Letting $g_\infty = m_n(T^{-1/2})$ gives the lemma.
\end{proof}

Here is the quantitative Sturm bound.
\begin{theorem}[Quantitative Sturm bound] Suppose $\varphi(g)$ is a cuspidal automorphic form on $\Sp_{2n}(\A)$ or its double cover corresponding to a holomorphic modular form of weight $\ell \geq 0$.  Let $M \geq 1$ be a positive integer so that $\beta_T(k) \neq 0$ implies $T \in M^{-1} S_n(\Z)^\vee$ for all $T$ and all $k \in K_f = \Sp_{2n}(\widehat{\Z})$ or $\widetilde{K}_f$.  Suppose $\epsilon \geq 0$ is such that that the normalized Fourier coefficients $\beta_T(g_f)$ satisfy $|\beta_T(k)| \leq \epsilon$ for all $k \in K_f$ and all $T$ with $\det(T) \leq R$, where
	\[R^{1/n} = \frac{2}{\pi n \epsilon_n} \log(2 e^{2\pi n} D_{\ell,n,\epsilon_n} M^\alpha).\]
Then 
\[|\varphi(g)| \leq \epsilon \left(2 e^{2\pi n} D_{\ell,n,\epsilon_n} M^{\alpha}\right).\]
In particular, 
\[|\beta_T(g_f)| \leq \epsilon \cdot \left(2 e^{4\pi n} D_{\ell,n,\epsilon_n} M^{\alpha}\right)\]
for all $T$ and all $g_f \in \Sp_{2n}(\A_f)$.
\end{theorem}
\begin{proof} We write out the proof in the case of the linear group $\Sp_{2n}$.  The proof in the case of the double cover is identical.
	
Let $g_{*} \in \Sp_{2n}(\A)$ be such that $|\varphi(g)|$ attains its maximum at $g_{*}$.  Let $L = |\varphi(g_*)|$.  By Corollary \ref{cor:adelicRedTh}, we can assume $g_{*} = g_{\infty} k \in \mathcal{S}(\epsilon_n) K_f$.  Write $y = Im(g_\infty \cdot i)$.  By Lemma \ref{lem:FCbound}, we obtain
\begin{align*}
L = |\varphi(g_{*})|  &\leq \sum_{T \in M^{-1} S_n(\Z)^\vee, T > 0}{|\beta_T(k)| \det(Ty)^{\ell/2} e^{-2\pi(T,y)}} \\
&\leq e^{2\pi n}( \epsilon S_{\ell}(M,y) + L T_{\ell}(R,M,y)).
\end{align*}
Applying Lemma \ref{lem:InfSums} gives
\[ L \leq e^{2\pi n} D_{\ell,n,\epsilon_n} M^\alpha (\epsilon + L e^{-\pi \epsilon_n n R^{1/n}/2}).\]
The constant $R$ is chosen so that 
\[e^{2\pi n} D_{\ell,n,\epsilon_n} M^\alpha e^{-\pi \epsilon_n n R^{1/n}/2} = \frac{1}{2}.\]
Rearranging the inequality gives the theorem.
\end{proof}

\section{Automatic convergence}
In this section, we prove Theorem \ref{thm:autConv}, the automatic convergence theorem.  We will need some prepatory lemmas.

\subsection{Prepatory lemmas}
	\begin{lemma}\label{lem:volBound1} Let $U_{r,g} \subseteq X(\A_f)$ be an open compact subset such that $a_{\diag(t,T')}(xrg) \neq 0$ implies $x \in U_{r,g}$.  Let $B$ be a positive real number so that $|a_{\diag(t,T')}(xrg)| \leq B$ for all $x \in X(\A_f)$.  Then $|a_{t,T'}(r,g,\phi)| \leq ||\phi||_{L^2} \cdot B \cdot \mathrm{vol}(U_{r,g})^{1/2}$.
\end{lemma}
\begin{proof} The proof is identical to \cite[Lemma 13.4]{pollackAutConvExc}.
\end{proof}

Conversely, we can bound the $a_{\diag(t,T')}(xg)$ in terms of the $a_{t,T'}(r,g,\phi)$.  
\begin{lemma}\label{lem:abarfromFJ} Suppose $B_{t,T',g}'> 0$ is a constant so that $|a_{t,T'}(1,g,\phi)| \leq B_{t,T',g}' \cdot ||\phi||_{L^2}$ for all $\phi \in S(X(\A_f))$.  Suppose $V_{t,T',g} \subseteq X(\A_f)$ is a compact open subgroup with the property that $a_{\diag(t,T')}(x v g)  = a_{\diag(t,T')}(x g) $ if $v \in V_{t,T',g}$ and $x \in X(\A_f)$.  Then $|a_{\diag(t,T')}(xg)| \leq B_{t,T',g}' \cdot \mathrm{vol}(V_{t,T',g})^{-1/2}$.
\end{lemma}
\begin{proof}  The proof is identical to \cite[Lemma 13.5]{pollackAutConvExc}.
\end{proof}

\subsection{Proof of automatic convergence}
To prove that the $a_T(g_f)$ grow polynomially for all $g_f$, it suffices to check it for $g_f \in K_f$.  Indeed, this follows from the $P$ symmetries and from the decomposition $\Sp_{2n}(\A_f) = N_P(\A_f) M_P(\Q) K_f$, where $P = M_P N_P$ is the Siegel parabolic subgroup.

To prove the polynomial growth, we will proceed by induction.  To set up the induction, let $\delta > 1$ be a positive real number, to be determined, and let $D_0 > 0$ be a sufficiently large real number.  If $T \in S_{n}(\Q)$ is positive-definite, then $\det(T) \leq D_0^{\delta^n}$ for some $n = 0,1,2,\ldots$.  We will show that, there are positive constants $Q, E >0$ so that if $D_0^{\delta^{n-1}} \leq \det(T) \leq D_0^{\delta^n}$, then 
\[|\beta_T(k)| \leq Q (1 \cdot D_0 \cdot D_0^\delta \cdots D_0^{\delta^{n-1}})^E = Q D_0^{E \cdot \frac{\delta^{n}-1}{\delta -1}} \leq Q \det(T)^{ \delta E/(\delta-1)}.\]
We will prove the first inequality by induction on $n$.  The second inequality shows that the growth is polynomial.

To begin, first note that there is a lattice $\Lambda \subseteq S_n(\Q)$ so that $a_T(k) \neq 0$ and $k \in K_f$ implies $T \in \Lambda$.  Indeed, this follows from the right $U$-invariance of the $a_T$'s and the fact that they are Siegel Whittaker functions.  See, e.g., \cite[Lemma 9.1]{pollackAutConvExc}.  Without loss of generality, we can assume that $U$ is normal in $K_f$ and $\Lambda$ is $\Gamma_U := \SL_n(\Q) \cap U$ invariant.  There are only finitely many $\Gamma_U$ orbits on the elements $T \in \Lambda$ with $\det(T)$ bounded.  See, e.g., \cite[Lemma 13.2]{pollackAutConvExc}.  Moreover, $|a_{T \cdot \gamma}(k)| = |a_{T}(k)|$ if $\gamma \in \Gamma_U$.  Thus, the base case of our induction can be satisfied for any $D_0$, for some $Q$ depending on $D_0$.

We now proceed with the induction step.  For ease of notation, if $n \geq 1$ and $D_0^{\delta^{n-1}} \leq D < D_0^{\delta^n}$, we set
\[ f(D) = Q (1 \cdot D_0 \cdot D_0^\delta \cdots D_0^{\delta^{n-1}})^E = Q D_0^{E \cdot \frac{\delta^{n}-1}{\delta -1}}.\]
We suppose that $|\beta_T(k)| \leq Q f(\det(T))$ for all $T$ with $\det(T) < D_0^{\delta^n}$.  

\begin{claim}\label{claim:Apply} Let $D_1 = D_0^{\delta^n}$.  Suppose $\epsilon > 0$ is small, $D_0$ is sufficiently large, and $t < D_1^{1-\epsilon}$.  Suppose $x \in X(\A_f)$.  Then there are positive constants $C'$ and $\alpha$ so that $|a_{\diag(t,T')}(xk)| \leq Q C'  \det(T')^{\ell/2} t^{\alpha'} f(D_1)$ for all $k \in K_f$.
\end{claim}
\begin{proof}
We apply the quantitative Sturm bound for cusp forms of weight $\ell-\frac{1}{2}$ on $\widetilde{\Sp}_{2n-2}(\A)$.  Fix $\phi \in S(X(\A_f))$.  The Fourier coefficients of our cusp form are 
\[a_{t,T'}(k_1,k,\phi) = \int_{X(\A_f)}{\omega(k_1)\phi(x) a_{\diag(t,T')}(xk_1 k)\,dx}.\]
The constant $M$ of the quantitative Sturm bound can be taken to be proportional to $t$.  Applying Lemma \ref{lem:volBound1}, we see that the normalized Fourier coefficients are bounded by a constant times $||\phi||^2 t^{\ell/2} \det(T')^{1/2} f(D_1)$ if $t \det(T') < D_1$.  As $t < D_1^{1-\epsilon}$, we see that if $\det(T') < (a_1 + b_1 \log(t))^{n-1}$ then $t \det(T') < D_1$.  Here $a_1, b_1$ are fixed postive constants.  We can thus apply the quantitative Sturm bound.  Applying it, we find
\[ |a_{t,T'}(r,k,\phi)| \leq Q C ||\phi||^2 \det(T')^{\ell/2} t^{\alpha'} f(D_1)\]
for some positive constants $C$ and $\alpha'$.  We now apply Lemma \ref{lem:abarfromFJ} to finish the proof.
\end{proof}

Suppose now $T \in S_n(\Q)$ and $D_1 = D_0^{\delta^n} \leq \det(T) \leq D_1^{\delta}$.  We wish to bound $|\beta_{T}(k)|$ for $k \in K_f$.  By Theorem \ref{thm:minkowski1}, we can assume $t = T_{11} \leq C_n D_1^{\delta/n}$.  Thus, if $1 < \delta < n$, $t < D_1^{1-\epsilon}$, so we can apply Claim \ref{claim:Apply}.  We obtain $|\beta_T(k)| \leq Q D_1^E f(D_1)$ for some positive number $E> 0$.  This completes the induction, and with it, the automatic convegence theorem.

	\bibliography{nsfANT2020new}
	\bibliographystyle{amsalpha}
\end{document}